\documentclass[12pt]{article}
\usepackage{fullpage}
\usepackage{latexsym, comment}
\usepackage{url}
\usepackage{epsfig,enumerate}
\usepackage{amsmath,amsthm,amssymb}

\usepackage[usenames]{color}\usepackage{graphicx}
\definecolor{brown}{cmyk}{0, 0.72, 1, 0.45}
\definecolor{grey}{gray}{0.5}

\newcounter{rot}%\addtocounter{rot}{1}, \therot

  \def\d{\delta} 
\def\e{\epsilon}    
  \def\k{\kappa} 
     
 \def\m{\mu} \def\n{\nu} 
  \def\s{\sigma} 
 \def\om{\omega}

\newtheorem{maintheorem}{Theorem}

\newtheorem{conjecture}{Conjecture}
\newtheorem*{conjecture*}{Conjecture}
\newtheorem{theorem}{Theorem}[section]
\newtheorem*{theorem*}{Theorem}
\newtheorem{procedure}{Procedure}
\newtheorem{lemma}[theorem]{Lemma}

\newtheorem{definition}[theorem]{Definition}

\newcommand{\bfrac}[2]{\left(\frac{#1}{#2}\right)}

\newcommand{\rai}{\rightarrow \infty}

\newcommand{\set}[1]{\left\{#1\right\}}

\def\es{\emptyset}

\def\E{\mathbb{E}}

\def\P{\mathbb{P}}
\def\Pr{\mathbb{P}}

\newcommand{\scr}{\mathcal}
\newcommand{\eps}{\epsilon}

\newcommand{\floor}[1]{\left\lfloor #1 \right\rfloor}

\newcommand{\of}[1]{\left( #1 \right) }

\newcommand{\abs}[1]{\left| #1 \right|}

\newcommand{\sqbs}[1]{\left[ #1 \right]}

\renewcommand{\E}[1]{\mathbb{E}\sqbs{#1}}
\newcommand{\tbf}[1]{\textbf{#1}}
\renewcommand{\P}[1]{\mathbb{P}\left[ #1 \right]}
\renewcommand{\Pr}[1]{\mathbb{P}\left[ #1 \right]}

\allowdisplaybreaks[1]

\newcommand{\ignore}[1]{}

\newcommand{\beq}[1]{\begin{equation}\label{#1}}
\newcommand{\eeq}{\end{equation}}

\newcommand{\Bin}{\operatorname{Bin}}

\newcommand{\vsi}{ V_{\s,i}}
\newcommand{\vsj}{V_{\s,j}}
\newcommand{\vs}[1]{V_{\s,{#1}}}
\newcommand{\epr}{$(\e,p)$-regular }
\newcommand{\gh}[1]{\widehat{G}_{#1}}
\newcommand{\nhv}{\widehat{N}_v}

\title{Packing tree factors in random and pseudo-random graphs}

\author{
Deepak Bal
\thanks{Department of Mathematics, Ryerson University, Toronto, ON, M5B 2K3, Canada.}
\and
Alan Frieze
\thanks{Department of Mathematical Sciences, Carnegie Mellon University, Pittsburgh, PA 15213.}
\thanks{Research supported in part by NSF Grant CCF2013110.}
\and
Michael Krivelevich
\thanks{School of Mathematical Sciences, Raymond and Beverly Sackler Faculty of Exact Sciences, Tel Aviv
University, Tel Aviv, 69978, Israel.
Research supported in part by a USA-Israel BSF grant and by a grant from the Israel Science Foundation.}
\and
Po-Shen Loh
\footnotemark[2]
\thanks{Research supported by NSF grant DMS-1201380, an NSA Young
Investigators Grant, and by a USA-Israel BSF grant.}
}

\begin{document}
\maketitle

\begin{abstract}
 For a fixed graph $H$ with $t$ vertices, an $H$-factor of a graph $G$ with
 $n$ vertices, where $t$ divides $n$, is a collection of vertex disjoint (not necessarily induced)
 copies of $H$ in $G$ covering all vertices of $G$. We prove that
 for a fixed tree $T$ on $t$ vertices and $\e>0$, the random graph $G_{n,p}$,
 with $n$ a multiple of $t$, with high probability contains a family of edge-disjoint
 $T$-factors covering all but an $\e$-fraction of its edges,
 as long as $\e^4np\gg \log^2n$. Assuming stronger
 divisibility conditions, the edge probability can be taken down to
 $p>\frac{C\log n}{n}$. A similar packing result is proved also for
 pseudo-random graphs, defined in terms of their degrees and
 co-degrees.

\end{abstract}

\section{Introduction}

Let $H$ be a graph with $t$ vertices and let $n$ be divisible by $t$. We
say that a graph $G=(V,E)$ with $n$ vertices has an \emph{$H$-factor} if
there exist vertex disjoint subgraphs (not necessarily induced) of $G$,
$H_1,\ldots,H_{n/t}$, which are all isomorphic to $H$. Note that the vertex
disjointness implies that the vertex set of the $H$-factor, $H_1 \cup
\cdots \cup H_{n/t}$, is equal to $V$.  $H$-factors have been an important
object in the study of random graphs.  Indeed, the most basic instance, a
$K_2$-factor, corresponds to a perfect matching. Erd\H{o}s and R\'enyi
\cite{ER66} proved in 1966 that if $p=\frac{\log n + \om}{n}$, with
$\om\rai$ and $n$ even, then  the Erd\H{o}s-R\'enyi-Gilbert random graph,
$G_{n,p}$ has a perfect matching \tbf{whp}\footnote{An event $\scr{E}_n$
occurs \emph{with high probability}, or \tbf{whp}, if
$\lim_{n\rai}\Pr{\scr{E}_n} = 1$.}. In 1981, Shamir and Upfal \cite{SU81}
proved a general result which implies that if $p=\frac{\log n +
(r-1)\log\log n +\om}{n}$, with $\om\rai$ arbitrarily slowly, and $n$ even,
then $G_{n,p}$ contains $r$ edge-disjoint perfect matchings \tbf{whp}. In
this range of $p$, the minimum degree of $G_{n,p}$ is $r$ \tbf{whp}, so the
result is optimal.  \L uczak and Ruci\'nski \cite{LR91}, as a corollary of
a more technical result, proved that for any tree $T$, if $p=\frac{\log n +
\omega}{n}$ with $\om\rai$, and $n$ is divisible by $\abs{T}$, then
$G_{n,p}$ has a $T$-factor \tbf{whp}. As an analogue to the theorem of
Shamir and Upfal, Kurkowiak \cite{K99} proved that if $p=\frac{\log n +
(r-1)\log\log n +\om}{n}$ with $\om\rai$, then $G_{n,p}$ contains $r$
edge-disjoint $T$-factors \tbf{whp}.

The study of optimal and near-optimal packings of spanning objects in
graphs and hypergraphs is an area of much active research.  Recently, the
case of Hamilton cycles (simple spanning cycles) has been the subject of
many papers. When $p = \frac{\log n  + (2r-1)\log\log n +\omega}{n}$ where
$\omega\rai$, Bollob\'as and Frieze \cite{BF85} proved in 1985 that
$G_{n,p}$ contains $r$ edge-disjoint Hamilton cycles \tbf{whp}.  In
\cite{FK05}, Frieze and Krivelevich conjectured that for any $0 < p = p(n)
\le 1$, $G_{n,p}$ contains $\floor{\d/2}$ edge-disjoint Hamilton cycles
\tbf{whp}, where $\d$ represents the minimum degree. The conjecture was
solved in the series of papers \cite{FK08},\cite{KKOExact},\cite{KS11} and
\cite{KO12}. In intermediate papers such as \cite{FK05} and \cite{KKO12},
the notion of approximate or almost optimal packings was studied. The
results in these papers state that for certain ranges for $p$, all but a
vanishing fraction of the edges of $G_{n,p}$ can be covered with
edge-disjoint Hamilton cycles.
%Such results were also obtained for random and pseudo-random hypergraphs
%in \cite{FK12}, \cite{FKL12} and \cite{BF12}.

In this work, we investigate when all but a vanishing fraction of
the edges of random and pseudo-random graphs be covered with
edge-disjoint $T$-factors, for a fixed tree $T$.  We begin by
introducing the notion of pseudo-randomness which we will use in
this paper.
\begin{definition}
 Let $G=(V,E)$ be a graph with $n$ vertices. We say $G$ is
 $\boldsymbol{(\e,p)}$\textbf{-regular} if the following 2 conditions hold:
 \vspace{-.1in}
\begin{itemize}
 \item $d(v) \ge (1-\eps)np$ for every vertex $v$.
\vspace{-.1in}
 \item $d(u,v)\le (1+\eps)np^2$ for every pair of distinct vertices $u$ and $v$.
\end{itemize}
\end{definition}

\noindent Here, $d(v)$ denotes the degree of vertex $v$, and
$d(u,v)$ denotes the co-degree of $u$ and $v$, i.e., the number of
neighbors common to both $u$ and $v$. We will also write $d_S(v)$
and $d_S(v,w)$ to refer to the degree and the co-degree into a set
$S$ of vertices.  Our pseudo-randomness conditions are localized,
and this is in part necessary because we are packing spanning
structures. Nevertheless, these conditions are satisfied \tbf{whp}
by $G_{n,p}$, for appropriately chosen $\eps$. Indeed, by the
standard Chernoff bound, stated as Theorem \ref{chernoff} in the
next section:
\begin{align*}
  &\Pr{G_{n,p}\textrm{ is not }(\e,p)\textrm{-regular}}  \\
 &< n\Pr{\Bin\sqbs{n-1,p} < (1-\e)np} + n^2\Pr{\Bin\sqbs{n-2,p^2} > (1+\e)np^2} \\
 &= o(1) \,,
\end{align*}
as long as $\e^2 n p^2 \gg \log n$.  (In this paper, we will write $A_n \gg
B_n$ when $A_n / B_n \rai$ with $n$.)   Our theorem for \epr graphs is then
as follows.
\begin{maintheorem}\label{mainpseudo}
Let $T$ be a fixed tree with $t$ vertices, and let $G$ be an
$(\e,p)$-regular graph on $n$ vertices, with $n$ a multiple of $t$. If
$\e,n$ and $p$ satisfy $\e^{6}np^4 \gg \log^3n$ then for $n$ sufficiently
large, $G$ contains a collection of edge-disjoint $T$-factors covering all
but $2\e^{1/3}$-fraction of its edges.
\end{maintheorem}

For random graphs, we have two results.  Let $\scr{P}(\e)$ be the graph
property that all but an $O(\e)$-fraction of the edges of a graph may be
covered by a collection of edge-disjoint $T$-factors. Direct application of
the pseudo-random result establishes that as long as the average degree
$np$ is above a certain power of $n$, random graphs $G_{n,p}$ can be almost
packed with tree-factors.  That initial range is not optimal, and resembles
the barrier which was hit during the investigation of Hamilton cycle
packing in random structures (see, e.g., \cite{BF12}, \cite{FK05}, and
\cite{FKL12}).  Using additional properties of $G_{n,p}$, we are able to
push the result to smaller $p$.

\begin{maintheorem}\label{randomthmlogsq}
  Let $T$ be a fixed tree with $t$ vertices. If $\e,n$, and $p$ satisfy
  $\e^4 np \gg \log^2 n$, then  $G_{n,p}$, with $n$ a multiple of $t$,
  satisfies $\scr{P}(\e)$ \tbf{whp}.
\end{maintheorem}

This range of $p$ ($\gg \frac{\log^2 n}{n}$) is still probably not optimal,
and in the context of Hamilton cycle packing, it took further developments
to remove the last logarithmic factors.  For tree-factor packing, however,
it turns out that we can circumvent this obstacle.  In the following
theorem, we improve the range of $p$ to asymptotically best possible,
subject to an additional divisibility condition on $n$, which we suspect to
be an artifact of our proof technique.
\begin{maintheorem}\label{randomthmlog}
  Given any $t$-vertex tree $T$ and any positive real $\epsilon$, there
  exists an integer $\tau_0$ such that for any $\tau \geq \tau_0$
  satisfying $t \mid \tau$, there is a real constant $C$ such that for $p >
  \frac{C \log n}{n}$, the random graph $G_{n,p}$ satisfies
  $\scr{P}(\epsilon)$ \tbf{whp} for $\tau \mid n$.
%\[\lim_{\stackrel{n \rai}{tK | n}}\P{G_{n,p}\textrm{ satisfies }
%\scr{P}(\e)} =1.\]
\end{maintheorem}

The complexity of the above result stems from the fact that it is stated in
greater generality.  Indeed, note that if one applies it with the
particular choice $\tau = \tau_0$, then the conclusion is that there is a
real $C_0$ such that $G_{n,p}$ satisfies $\scr{P}(\epsilon)$ \tbf{whp} for
$\tau_0 \mid n$, when $p > \frac{C_0 \log n}{n}$.  This is within a factor
($C_0$) of the best possible result, and the divisibility condition is also
off by a factor (ideally, it would only require $t \mid n$).  Although it
may be more challenging to eliminate $C_0$, we conjecture that perhaps it
may not be as difficult to relax the divisibility condition.
\begin{conjecture}
  Given any $t$-vertex tree $T$ and any positive real $\epsilon$, there
  exists a real constant $C$ such that for $p > \frac{C \log n}{n}$, the
  random graph $G_{n,p}$ satisfies $\scr{P}(\epsilon)$ \tbf{whp} for $t
  \mid n$.
\end{conjecture}

Throughout our exposition, we will implicitly assume that $\epsilon$ is
sufficiently small and $n$ is sufficiently large.  The following (standard)
asymptotic notation will be utilized extensively.  For two functions $f(n)$
and $g(n)$, we write $f(n) = o(g(n))$ if $\lim_{n \rightarrow \infty}
f(n)/g(n) = 0$, and $f(n) = O(g(n))$ if there exists a constant $M$ such
that $|f(n)| \leq M|g(n)|$ for all sufficiently large $n$.  We also write
$f(n) = \Theta(g(n))$ if both $f(n) = O(g(n))$ and $g(n) = O(f(n))$ are
satisfied.  We also write $A=(1\pm\e)B$ to mean $(1-\e)B \le A \le
(1+\e)B.$ All logarithms will be in base $e \approx 2.718$.

\section{Concentration inequalities}

For the reader's convenience, we record in this section the two
large-deviation bounds which we will use in this paper.  We will appeal to
the following version of the Chernoff bound, which can be found, for
example, as Corollary 2.3 in the book by Janson, \L uczak, and Ruci\'nski
\cite{JLR00}.
\begin{theorem}\label{chernoff}
 Let $X$ be a binomial random variable with mean $\m$, and let $0 < \eps
 <1$. Then
 \[\Pr{\abs{X-\mu} >\eps\m} \le 2e^{-\e^2\m/3} \,.\]
\end{theorem}

The previous result establishes concentration of a random variable
defined over a product space.  In this paper, we will also encounter
a particular non-product space.  For that, we use the following
concentration bound which applies in the setting where the
probability space is the uniform distribution over permutations of
$n$ elements. For a proof, we refer the reader to \cite{FP04} or
\cite{McD98}.

\begin{theorem}\label{permconc}
  Let $X$ be a random variable determined by a uniformly random permutation
  on $n$ elements, and let $C$ be a real number.  Suppose that whenever
  $\s, \s' \in S_n$ differ by a single transposition, $\abs{X(\s) - X(\s')}
  \le C$. Then,
\[\P{\abs{X-\E{X}} \ge t} \le 2\exp\set{-\frac{2t^2}{C^2n}}.\]
\end{theorem}

\section{Proof of Theorem \ref{mainpseudo}}

Let $G=(V,E)$ be a graph on vertex set $V=[n]$, and suppose $t$ divides $n$
with $\nu=n/t$. Let $T=(V_T,E_T)$ be a fixed tree on vertex set $V_T=[t]$.
Let $\s$ be a permutation of $[n]$. Let $G_\s=(V_\s,E_\s)$ be the
$t$-partite subgraph of $G$ with vertex set
\[V_\s = V_{\s,1}\, \dot{\cup}\, V_{\s,2}\, \dot{\cup} \cdots \dot{\cup} \,V_{\s,t}\]
where $V_{\s,i} = \set{\s((i-1)\nu+1),\ldots,\s(i\nu)}$ for $i=1,\ldots,t.$ The edge set is defined as
\[E_\s = \set{(u,v)\in E\,:\,\exists i,j\in[t],\, \text{ with }(i,j)\in E_T \textrm{ and }u\in V_{\s,i}, v\in V_{\s,j}}.\]
In words, we use $\s$ to define a partition of $V$ into $t$ parts,
corresponding to the vertices of $T$, and we keep the edges of $G$
which connect two parts corresponding to the endpoints of an edge of
$T$. We also separately define $G_\s'$ as the subgraph of $G$ where
we keep edges between all pairs $(\vsi,\vsj)$, but still delete the
edges within the $\vsi$.

The resulting $G_\s$ looks like a ``blown-up'' version of $T$.  We call a
pair $(V_{\s,i},V_{\s,j})$ a \emph{super-edge} if $(i,j)\in E_T$, and we
say that it is an \emph{$(\e,p)$-regular pair} if
\begin{itemize}
 \item for all $v\in\vsi$ and $w\in\vsj$, $d_{\vsj}(v),d_{\vsi}(w) \ge
   (1-\e)\nu p$, and
 \item for all $v,w \in \vsi$ and $v',w'\in\vsj$, $d_{\vsj}(v,w), d_{\vsi}(v',w') \le (1+\e)\nu p^2.$
\end{itemize}

We say that $G_\s$ is an \emph{$(\e,p)$-regular blow up of $T$} if every
super-edge is an $(\e,p)$-regular pair.  Conveniently, if we take an
$(\e,p)$-regular graph $G$, and uniformly sample a random permutation $\s$
of $[n]$, then we typically preserve the regularity across super-edges in
$G_\s$.  Formally, we have:
\begin{lemma}\label{Gsigmaisablowup}
Let $G$ be an $(\e,p)$-regular graph on $n$ vertices with $n$ divisible by
$t$ and $\nu=n/t$. Suppose that $\e^2np^4\gg \log n.$  Let $\s$ be a
uniformly random permutation on $n$ elements, and define $G_\s$ as
above. Then with probability $1-o(n^{-1})$, $G_\s$ is a $(2\e,p)$-regular
blow up of $T$.
%for every $i\in [t]$,
%\begin{itemize}
%\item $d_{V_{\s,j}}(v) = (1\pm 2\e)\nu p$ for every $v\in V_{\s,i}$ and every super-edge $(V_{\s,i}, V_{\s,j}).$
%\item $d_{V_{\s,j}}(v,w)$ = $(1\pm 2\e)\nu p^2$ for every $v,w \in V_{\s,i}$ and every super-edge $(V_{\s,i}, V_{\s,j}).$
%\end{itemize}
%By definition of $G_\s$, all degrees within the $\vsi$ and across non-super-edges are zero.
\end{lemma}

\begin{proof}
 We will show that all pairs $(\vsi,\vsj)$ in $G_\s'$ are $(2\e,p)$-regular
 pairs, which obviously implies the result for  $G_\s$ since $G_\s$ and
 $G_\s'$ agree on super-edges.  We first show that all degrees are typically correct.
 Let $v$ be an arbitrary vertex and expose only the position of $v$ under
 $\s$. Suppose this reveals that $v\in\vsi$. Consider the pair
 $(\vsi,\vsj)$ in $G_\s'$ for any $j\neq i$.  Let $N_v := d_{\vsj}(v)$ and
 note that this is a random variable whose randomness comes from the
 permutation $\s$. Conditioned on the position of $v$, $\s$ is a uniform
 random permutation on the $n-1$ remaining vertices and so every other
 vertex has probability $\frac{n/t}{n-1}=\frac{\nu}{n-1}$ of being in
 $\vsj.$  We also know that $d_G(v) \ge (1-\e)np$ by $(\e,p)$-regularity,
 so $\E{N_v} \ge (1 - 1.5\e)\nu p$.

For concentration we apply Theorem \ref{permconc} to $N_v$. Note that transposing two elements of $\s$ can only change $N_v$ by at most $1$. So the probability that $N_v$ differs from its mean by more than $.5\e\nu p$ is bounded above by
\[2\exp\set{-\frac{2\of{.5\e\nu p}^2}{n-1}} = o(n^{-K})\]
for any positive constant $K$ as long as $\e^2 n p^2 \gg \log n$.
So taking a union bound over all vertices and choices of $j$ for $\vs{j}$, we have the degree conclusion of the lemma.

For co-degrees, we proceed similarly.  Let $v$ and $w$ be arbitrary
vertices and expose the positions of these two vertices under $\s$.
Suppose this reveals that $v\in \vsi$ and $w\in \vsj$. Let $k\in[t]$
be distinct from $i$ and $j$. Note that we are really only concerned
with the case when $i=j$ and $(\vsi, \vs{k})$ is a super-edge, but
this does not matter much. Let $N_{v,w}$ be the co-degree of $v$ and
$w$ into $\vs{k}$ in $G_\s'.$  Conditioned on the positions of $v$
and $w$, $\s$ is a uniform random permutation on the $n-2$ remaining
vertices, so every other vertex has probability $\frac{\nu}{n-2}$ of
being in $\vs{k}$. Also $d_G(v,w)\le(1+\e)np^2$ by
$(\e,p)$-regularity, so $\E{N_{v,w}} \le (1+ 1.5\e)\nu p^2$.

Applying Theorem \ref{permconc}, we have $N_{v,w} \le (1+ 2\e)\nu p^2$ with probability at least $1-o(n^{-K})$ for any positive constant $K$ as long as $\e^2np^4\gg \log n.$ Taking a union bound over pairs of vertices and choices of $k$ for $\vs{k}$, we have the co-degree conclusion of the lemma.
\end{proof}

We now define a procedure for generating edge-disjoint subgraphs of an \epr
graph $G$, each of which looks something like a $G_\s$.

\begin{procedure}\label{proc1}
This procedure takes as input a graph $G=(V,E)$ on $n$ vertices with
$n$ divisible by $t$. Let \[r = \frac{30}{\e^2}\frac{t^2}{t-1}\log
n,\] and perform the following steps.
\begin{enumerate}
\item[\tbf{P1}] Generate $r$ independent uniformly random permutations $\s_1,\ldots,\s_r$ of
  $[n]$. Construct $G_1=G_{\s_1},\ldots,G_r=G_{\s_r}$
  as described above.
\item[\tbf{P2}] For each edge $e\in E$, let $L_e=\set{i\,:\,e\in G_i}$. If
  $L_e \neq \es$, select a uniformly random element $i$ of $L_e$, and label
  $e$ with $i$.
\item[\tbf{P3}] Let $\gh{i} = (\widehat{V}_i, \widehat{E}_i)$ be the subgraph of $G_i$ consisting of all edges which received label $i$.
\end{enumerate}
\end{procedure}

Note that the $\gh{i}$'s are now edge-disjoint by construction. Our goal
will be to prove that the $\gh{i}$'s have regularity properties similar to
those of the $G_i$'s, but with larger $\e$ and smaller $p$.

\begin{lemma}\label{everyEdgeIn}
 Run Procedure \ref{proc1} on an \epr graph $G$ with $n$ vertices, where
 $n$ is divisible by $t$. Then with probability $1-o(n^{-1})$, each edge
 $e\in G$ appears in $(1\pm\e)\k$ of the $G_i$'s, where
 \[\k=\frac{60}{\e^2}\log n = \frac{2(t-1)}{t^2}r.\]
\end{lemma}
\begin{proof}
 We first compute the probability $q$ that an edge $e=(v,w)$ appears in
 $G_\s$, when $\s$ is a uniformly random permutation. Then if we let $X_e$
 be the random variable counting the number of $G_i$'s which contain
 $e$, by the independence of the permutations $\s_1,\ldots,\s_r$, we have that
 $X_e$ is distributed as $\Bin[r,q]$.

 The edge $e$ appears in $G_\s$ if and only if $e$ is part of a super-edge.
 Let $(i,j)$ be a fixed edge of $T$. The probability that $e$ crosses
 $(\vsi,\vsj)$ is $\frac{2}{t^2}\frac{n}{n-1}$. To see this, expose the
 position of $v$ under $\s$.  Then $v$ lies in $\vsi$ or $\vsj$ with
 probability $2\cdot\frac{n/t}{n}$. Conditioning, on this, the probability
 that $w$ lands in the other set is $\frac{n/t}{n-1}$. There are $t-1$
 edges of $T$, and the events corresponding to $e$ belonging to different
 super-edges are mutually disjoint. So we have
 \[q = \frac{2(t-1)}{t^2}\of{1+\frac{1}{n-1}}\,,\]
 which implies that
 \[\E{X_e} = rq=\frac{60}{\e^2}\log n\of{1+\frac{1}{n-1}}\,,\]
 and by
Theorem \ref{chernoff}, the probability that $X_e$ differs from its mean by more than $.5\e rq$
is bounded above by
\[\exp\set{-\frac{(.5\e)^2}{3}rq} = o(n^{-3}).\]
Hence with probability at least $1-o(n^{-1})$, every edge is contained in $(1\pm\e)\frac{60}{\e^2}\log n$ of the $G_i$'s.
\end{proof}

\begin{lemma}\label{ghatsaregood}
 Suppose that $\e^6np^4\gg\log^3n$. Then after Step $\tbf{P3}$ of Procedure \ref{proc1}, with probability $1-o(1)$, each $\gh{i}$ is a $(7\e,\frac{p}{\k})$-regular blow up of $T$.
\end{lemma}
\begin{proof}
  Since $r=\Theta(\log n/\e^2) = o(n)$, by our assumptions on $n,p$ and
  $\e$, we may assume that at the beginning of step
  \tbf{P3}, each of $G_1,\ldots,G_r$ is a $(2\e,p)$-regular blow up of $T$
  and each edge of $G$ appears in $(1\pm\e)\k$ of the $G_i$'s.

  Now consider a single $\gh{i}.$ We will show that with sufficiently high
  probability, this is a $(7\e,\frac{p}{\k})$-regular blow up of $T$. This
  means we must show that across super-edges, all degrees and co-degrees
  are very typically correct.  The super-edges of $\gh{i}$ are the same as those in $G_i$,
  but the edge density is lower by a factor of approximately $\k$ since
  each edge of $G_i$ chose to be in $\gh{i}$ with probability approximately
  $1/\k$.

  Let $v$ be an arbitrary vertex of $\gh{i}$, and let $\nhv$ represent the
  degree of $v$ across a super-edge in $\gh{i}$. If we let $Z_v$ represent
  the degree of $v$ across this super-edge in $G_i$, then since $G_i$ is a
  $(2\e,p)$-regular blow up of $T$, we have that $Z_v \ge (1 -2\e)\nu p$.
  Each of these $Z_v$ edges is included in $\gh{i}$ independently with
  probability $\frac{1}{(1\pm\e)\k}$. So $\nhv$ stochastically dominates
  the distribution
  \[\Bin\sqbs{Z_v, \frac{1}{(1+\e)\k}}.\]
  Thus
  \[\E{\nhv} \ge (1 - 4\e)\nu \frac{p}{\k},\]
  and Theorem \ref{chernoff} tells us that
  \[\nhv \ge (1 - 6\e)\nu \frac{p}{\k}\]
  with probability $1-o(n^{-K})$ for any positive constant $K$ as long as
  \[\e^2n\frac{p}{\k} \gg \log n \iff \e^4np \gg \log^2n. \]

  \newcommand{\nhvw}{\widehat{N}_{v,w}}

  Now for co-degrees, we let $v,w$ be vertices in $\gh{i}$ and let $\nhvw$
  represent the co-degree of $v$ and $w$ across a super-edge in $\gh{i}$.
  We let $Z_{v,w}$ represent the co-degree of $v$ and $w$ across this
  super-edge in $G_i$. Applying $(2\e,p)$-regularity of $G_i$, we have that
  $Z_{v,w} \le (1+2\e)\nu p^2$.  Each of these vertices remains a common
  neighbor of $v$ and $w$ in $\gh{i}$ with probability
  $\bfrac{1}{(1\pm\e)\k}^2$, so $\nhvw$ is stochastically dominated by
  \[\Bin\sqbs{Z_{v,w}, \bfrac{1}{(1-\e)\k}^2 }.\]
  Thus
  \[\E{\nhvw} \le (1+ 5\e)\nu \bfrac{p}{\k}^2\]
  and Theorem \ref{chernoff} gives that
  \[\nhvw  \le  (1 + 7\e)\nu\bfrac{p}{\k}^2\]
  with probability $1-o(n^{-K})$ for any positive constant $K$ as long as
  \[\e^2n\bfrac{p}{\k}^2 \gg \log n \iff \e^6np^2\gg \log^3 n.\]
  Taking a union bound over choices of vertices, super-edges and $\gh{i}$,
  we conclude that with probability $1-o(1)$ each $\gh{i}$ is a $(7\e,\frac{p}{\k})$-regular blow-up
  of $T$.
\end{proof}

To prove Theorem \ref{mainpseudo}, we will apply the following result of
Frieze and Krivelevich, which shows that any $(\e,p)$-regular pair can have
almost all of its edges covered by edge-disjoint perfect matchings.

\begin{lemma}[Frieze, Krivelevich \cite{FK12}]\label{packmatch:pseudo}
 Suppose $(A,B)$ is an $(\eta,d)$-regular pair with $\abs{A}=\abs{B}=\nu$  and $\eta^{4/3}d^2\nu \gg 1$ for some small value $\eta\ll 1$. Then $(A,B)$
contains a collection of $(1-\eta^{1/3})d\nu$ edge-disjoint perfect matchings.
\end{lemma}
%---original statement of packmatch:pseudo-----------------------------------------
%\begin{lemma}[Frieze, Krivelevich \cite{FK}]\label{packmatch:pseudo}
% Let $G$ be a bipartite graph with vertex set $A \cup B$ where $\abs{A} = \abs{B} = n$. Suppose
%that the minimum degree in G is at least $(1-\e)dn$ and the maximum co-degree of two vertices
%is at most $(1+\e)d^2n$ for some small value $\e \ll 1$. Suppose further that $\e^{4/3}d^2n \gg 1$. Then $G$
%contains a collection of $(1-\e^{1/3})dn$ edge-disjoint perfect matchings.
%\end{lemma}

\begin{proof}[Proof of Theorem \ref{mainpseudo}.]
  Let $G$ be an $(\e,p)$-regular graph on $n$ vertices with $\e^{6}np^4 \gg
  \log^3n$.  Apply Procedure \ref{proc1}.  Our conditions on $\e,p$, and
  $n$ allow us to apply Lemma \ref{ghatsaregood} and to conclude that at the
  end of Step \tbf{P3} of Procedure \ref{proc1}, every $\gh{i}$ is a
  $(7\e,\frac{p}{\k})$-regular blow up of $T$. We also have that each
  edge of $G$
  appears in \emph{exactly}\/ one of the $\gh{i}$.

  Consider one of the $\gh{i}$, and call its $t-1$ super-edges
  $Q_1,\ldots,Q_{t-1}$. On each of these super-edges,  we apply Lemma
  \ref{packmatch:pseudo} with $\nu=n/t$, $\eta=7\e$ and $d=\frac{p}{\k}$.
  Then we have that each $Q_j$ contains a collection of edge-disjoint
  perfect matchings $\scr{M}_j$ of size at least
  \[s:=(1-(7\e)^{1/3})\frac{p}{\k}\frac{n}{t}.\]
  Now select arbitrary matchings $M_1 \in \scr{M}_1, M_2\in\scr{M}_2,\ldots,M_{t-1}\in\scr{M}_{t-1}.$
  Observe that
  \[M_1\cup M_2\cup\cdots\cup M_{t-1}\]
  is a $T$ factor since the super-edge structure of $\gh{i}$ is isomorphic
  to $T$. We may thus extract at least $s$ edge-disjoint $T$-factors from
  $\gh{i}$.  Indeed, we may do this for each of $\gh{1},\ldots,\gh{r}$.
  Tree factors extracted from distinct $\gh{i}$ are edge-disjoint.

  In total, the number of edges covered by these tree factors is at least
  \[s\cdot\frac{n}{t}\cdot(t-1)\cdot r = (1-(7\e)^{1/3})\frac{n^2}{2}p\,,\]
  while the $(\e,p)$-regularity of $G$ tells us that $G$ had at most
  $(1+\e)\frac{n^2}{2}p$ edges total.  So, the total fraction covered is at
  least
  \[\frac{1-(7\e)^{1/3}}{1+\e} \ge 1-2\e^{1/3}\,,\]
  as long as $\e$ is sufficiently small.
\end{proof}

\section{Proof of Theorem \ref{randomthmlogsq}}

Direct application of Theorem \ref{mainpseudo} for pseudo-random
graphs gives a packing result for random graphs with $p \gg n^{-1/4}
\log^{3/4} n$. In this section, we use additional properties of
random graphs to improve the range of $p$ to $p \gg \frac{\log^2
n}{n}$.  For this, we will apply the following result, which is an
analogue of Lemma \ref{packmatch:pseudo} for the fully random graph
setting.

\begin{lemma}[Frieze, Krivelevich \cite{FK12}]\label{packmatch:random}
  Let $G$ be a random bipartite graph with sides $A, B$ of size $\abs{A} =
  \abs{B} = \n$, where each edge appears independently with probability at
  least $p = p(\n)$. Assume that $p(\n) \gg \frac{\log \n}{\n}$. Then with
  probability $1 - o(\n^{-1})$, $G$ contains a family of $(1 - \d)\n p$
  edge-disjoint perfect matchings, where
  \[\d = \bfrac{16\log \n}{\n p}^{1/2}.\]
\end{lemma}

\begin{proof}[Proof of Theorem \ref{randomthmlogsq}]
  The proof of this theorem is essentially identical to that in the
  previous section. Run Procedure \ref{proc1} on a random graph $G=G_{n,p}$
  with $n$ a multiple of $t$, and let $\nu = n/t$.  Since $G$ is random,
  after Step \tbf{P1},  each super-edge of each $G_i$ is a copy of
  $B_{\nu,\nu,p}$, the random bipartite graph with parts of size $\nu$ and
  edge probability $p$.  The proof of Lemma \ref{everyEdgeIn} applied to
  $G_{n,p}$ instead of an \epr graph gives us that with probability
  $1-o(n^{-1})$, each edge appears in $(1\pm \e)\k$ of the $G_i$'s where
  $\k = \frac{60}{\e^2}\log n$.

  Conditioning on this, we see that after Step \tbf{P3} of Procedure
  \ref{proc1}, in a particular $\gh{i}$, across a super-edge $(A,B)$, each
  pair $(a,b), a\in A, b\in B$ is an edge of $\gh{i}$ with probability at
  least
  \begin{displaymath}
    q :=
    p\cdot\frac{1}{(1+\e)\k}
    \gg \frac{\log n}{n}
    = \Theta\left(\frac{\log\n}{\n}\right).
  \end{displaymath}
  So, we may apply Lemma \ref{packmatch:random} to each of the $(t-1)r$
  super-edges in the $\gh{i}$'s.  Since $(t-1)r \ll \nu$, we have that
  \tbf{whp}, each of the $(t-1)r$ super-edges satisfies the conclusion of
  the lemma.

  Consider one of the $\gh{i}$'s and suppose that we call its $t-1$
  super-edges $Q_1,\ldots,Q_{t-1}$. Then we have that each $Q_j$ contains
  a collection of edge-disjoint perfect matchings $\scr{M}_j$ of size at
  least
  \[s:=(1-\d)q\n.\]
  Note that $\d$ is bounded by
  \[\bfrac{16\log\nu}{\nu \frac{p}{(1+\e)\k}}^{1/2} = O(\e).\]
  As before, selecting arbitrary matchings $M_1 \in \scr{M}_1,
  M_2\in\scr{M}_2,\ldots,M_{t-1}\in\scr{M}_{t-1}$ gives a $T$-factor
  \[M_1\cup M_2\cup\cdots\cup M_{t-1}.\]
 We may thus extract at least $s$ edge-disjoint $T$-factors from $\gh{i}$ and do this for each of $\gh{1},\ldots,\gh{r}$. In total, the number of edges covered will be at least
\[s\cdot\n\cdot(t-1)\cdot r = (1-O(\e))\frac{n^2}{2}p\]
which is the desired number of edges since the graph has at most
$(1+\e)\binom{n}{2}p$ edges total \tbf{whp} by Theorem \ref{chernoff}.

\end{proof}

\section{Proof of Theorem \ref{randomthmlog}}

We conclude by introducing a different argument which ``bootstraps''
Theorem \ref{mainpseudo} to drive the range of $p$ all the way down to
about $\frac{\log n}{n}$.  Note that this is essentially the limit, because
for $p$ below $\frac{\log n}{n}$, the random graph typically contains
isolated vertices, and therefore finding even a single $T$-factor would be
impossible.

\begin{proof}[Proof of Theorem \ref{randomthmlog}]
  Let $\tau_1$ be the smallest value of $\tau$ for which Theorem 1 applies
  for $(\epsilon^3, 1)$-regular graphs on $\tau$ vertices.  Let $\tau_0 =
  \max\{\tau_1, \epsilon^{-3}\}$, and assume that $\tau \geq \tau_0$, with
  $t \mid \tau$.  Define $C = \tau \epsilon^{-2}$.  The idea of the proof
  is to first split the vertex set of $G\sim G_{n,p}$ into $\tau$ parts
  $V_1 \cup \cdots \cup V_\tau$ of size $\ell$ each.  We then think of
  these parts as vertices of the complete graph on $\tau$ vertices and note
  that the complete graph $K_\tau$ is a $(\d,1)$-regular graph for any $\d
  \geq 1/\tau$. We apply Theorem \ref{mainpseudo} to find a collection of
  edge-disjoint $T$-factors which cover almost all the edges of $K_\tau$.
  Each edge appearing in a $T$-factor of $K_\tau$ corresponds to a random
  bipartite graph $B_{\ell,\ell,p}$ in $G$. To each such bipartite graph we
  apply Lemma \ref{packmatch:random}. We must show that the total number of
  edges covered is at least $(1-O(\e))\binom{n}{2}p$ since by Theorem
  \ref{chernoff}, $G_{n,p}$ has at most $(1+\e)\binom{n}{2}p$ edges total
  \tbf{whp}.

  We now analyze this procedure quantitatively.  We fail to cover edges in
  three ways: when they are within a single $V_i$, when they are between a
  pair $(V_i, V_j)$ which is not covered by a $T$-factor, and when they are
  within a $T$-factor edge, but missed by Lemma \ref{packmatch:random}.  We
  must ensure that the total fraction missed is $O(\epsilon)$.  To this
  end, note that the first omission loses only at most
  $\frac{1}{\tau}$-fraction of the edges, while the second loses at most $2
  \delta^{1/3}$-fraction of the edges by Theorem \ref{mainpseudo}.
  Therefore, as long as $\tau > \tau_0 = \epsilon^{-3}$ (which implies that
  $K_\tau$ is $(\delta,1)$-regular with $\delta = \epsilon^3$), the total
  loss from the first two types is only $O(\epsilon)$.  Note that this
  bounded loss is completely deterministic.

  We control the third type of omission using the randomness in $G_{n,p}$.
  The bipartite graph between every pair $(V_i,V_j)$ covered by a
  $T$-factor in $K_\tau$ is a copy of the random bipartite graph
  $B_{\ell,\ell,p}$.  If we apply Lemma \ref{packmatch:random} to such a
  graph, we obtain a collection $\scr{M}_{i,j}$ of at least
  $\of{1-\bfrac{16\log\ell}{\ell p}^{1/2}}\ell p$ edge-disjoint matchings
  with probability $1-o(\ell^{-1})$.  Since we are proving that $G_{n,p}$
  satisfies $\scr{P}(\epsilon)$ \tbf{whp}, $\tau$ is a constant while $n
  \rightarrow \infty$, and therefore $\tau^2 \ll \ell$; a union bound then
  implies that \tbf{whp}, every pair $(V_i,V_j)$ from the $T$-factor of
  $K_\tau$ contains such a collection $\scr{M}_{i,j}$.  As in the proofs of
  our other two results, these perfect matchings combine to form
  $T$-factors of the full $n$-vertex graph.  It therefore remains only to
  show that the fractional loss can be kept below $O(\epsilon)$.  For
  this, we use $p > \frac{C \log n}{n}$, and simplify:
  \begin{displaymath}
    \bfrac{16\log\ell}{\ell p}^{1/2}
    < \bfrac{16\log\ell}{\ell \frac{C \log(\tau\ell)}{\tau \ell}}^{1/2}
    =O\left(\left( \frac{\tau}{C} \right)^{1/2}\right)
    = O(\epsilon) \,,
  \end{displaymath}
  since $C = \tau \epsilon^{-2}$.
\end{proof}

\end{document}